\documentclass[12pt]{amsart}
\usepackage{geometry}
\usepackage[mathscr]{euscript}
\usepackage{ amssymb }
\usepackage[comma, numbers]{natbib}
\usepackage{bm}
\usepackage{enumerate}
\usepackage[english]{babel}
\usepackage{commath}
\usepackage{tikz-cd}
\usepackage{graphicx}
\usepackage [autostyle, english = american]{csquotes}
\MakeOuterQuote{"}

\theoremstyle{plain}
\newtheorem{theorem}{Theorem}[section]

\theoremstyle{plain}
\newtheorem{lemma}[theorem]{Lemma}

\theoremstyle{plain}
\newtheorem{corollary}[theorem]{Corollary}

\theoremstyle{definition}
\newtheorem{definition}[theorem]{Definition}

\theoremstyle{plain}
\newtheorem{proposition}[theorem]{Proposition}

\theoremstyle{remark}

\theoremstyle{definition}
\newtheorem{example}[theorem]{Example}

\theoremstyle{plain}

\theoremstyle{plain}

\theoremstyle{plain}

\title[Proximity inductive dimension and Brouwer dimension]{Proximity inductive dimension and Brouwer dimension agree on compact Hausdorff spaces}

\author{Jeremy Siegert}
\address{University of Tennessee, Knoxville, USA} 
\email{jsiegert@vols.utk.edu}
\date{\today} 


\begin{document}
\maketitle

\begin{abstract}
	In this paper we show that the proximity inductive dimension defined by Isbell agrees with the Brouwer dimension originally described by Brouwer on the class of compact Hausdorff spaces. Consequently, Fedorchuk's example of a compact Hausdorff space whose Brouwer dimension exceeds its Lebesgue covering dimension is an example of a space whose proximity inductive dimension exceeds its proximity dimension as defined by Smirnov.
\end{abstract}

\section{Introduction}

Proximity spaces in their modern form were described during the early 1950's by Efremovi\v{c}, \cite{efremovic1},\cite{efremovic2}. Variations of the original definition can be found in \cite{proximityspaces}. The structure is meant to capture the notion of what it means for two subsets of a space to be "close". Every proximity space has a natural completely regular topological structure and the class of proximity spaces is placed somewhat neatly between topological spaces and uniform spaces in that every uniform space induces a proximity structure whose corresponding topology is the uniform topology. Likewise every proximity space is induced by at least one uniform structure (see \cite{proximityspaces}). The dimension theory of proximity spaces began when Smirnov defined the proximity dimension $\delta d$ of proximity spaces using $\delta$-coverings, \cite{Smirnov}. This dimension function serves as a proximity invariant analog of the covering dimension $dim$. In the case of compact Hausdorff spaces, whose topology is induced by a unique proximity, the dimensions $\delta d$ and $dim$ coincide. A proximity invariant inductive dimension would not be defined until Isbell defined the notion of a "freeing set" and subsequently the proximity inductive dimension $\delta Ind$ in \cite{proximityinductivedimension}. Isbell remarked in \cite{findimunifspaces2} and \cite{isbell} that he did not know of an instance where $\delta Ind$ and $\delta d$ did not agree. In this paper we will show that a space constructed by Fedorchuk in \cite{spacewithdifferingdimensions} has distinct $\delta Ind$ and $\delta d$. We do this by shown that $\delta Ind$ and the Brouwer dimension $Dg$ agree on the class of compact Hausdorff spaces. For the sake of self-containment we review the necessary preliminary definitions in sections \ref{proximity spaces} and \ref{brouwer dimension} before proceeding to our main results in section \ref{main results}. Throughout this paper we use the notation $\overline{A}$ and $int(A)$ for the closure and interior of a subset $A$ within a topological space $X$.

\section{Proximity Spaces and their dimensions}\label{proximity spaces}

In this section we will review the necessary definitions and results surrounding proximity spaces. The citations are not necessarily where the corresponding definitions or results first appeared, but where they can be easily found. These initial definitions and results about proximity spaces can be found in \cite{proximityspaces}.
\vspace{\baselineskip}

\begin{definition}\label{proximity space definition}
	Let $X$ be a set and $\delta$ a binary relation on $2^{X}$. The relation $\delta$ is said to be a {\bf proximity relation}, or simply a {\bf proximity} on $X$, if the following axioms are satisfied for all $A,B,C\subseteq X$:
	\begin{enumerate}
		\item $A\delta B$ if and only if $B\delta A$.
		\item $(A\cup B)\delta C$ if and only if $A\delta C$ or $B\delta C$.
		\item $A\delta B$ implies that $A\neq\emptyset$ and $B\neq\emptyset$.
		\item $A\cap B\neq\emptyset$ implies that $A\delta B$.
		\item $A\bar{\delta}B$ implies that there is an $E\subseteq X$ such that $A\bar{\delta}E$ and $(X\setminus E)\bar{\delta}B$.
	\end{enumerate} 
Where $A\bar{\delta}B$ is interpreted as "$A\delta B$ is not true". A pair $(X,\delta)$ where $X$ is a set and $\delta$ is a proximity on $X$ is called a {\bf proximity space}. If for a proximity space $(X,\delta)$ the relation $\delta$ satisfies the additional axiom that for all $x,y\in X$, $\{x\}\delta\{y\}$ if and only if $x=y$ we call the proximity $(X,\delta)$ {\bf separated}.
\end{definition}

As mentioned in the introduction every proximity space has a natural topological structure. This topology is defined in the following way:

\begin{proposition}
	If $(X,\delta)$ is a proximity space then the function $cl:2^{X}\rightarrow 2^{X}$ defined by
	\[cl(A):=\{x\in X\mid \{x\}\delta A\}\]
	
	is a closure operator on $X$. Moreover, the corresponding topology is Hausdorff if and only if $(X,\delta)$ is separated.
\end{proposition}

We will call the topology on a proximity space $(X,\delta)$ described above the {\bf topology induced by the proximity} $\delta$. A set $U\subseteq X$ is open in the induced topology if and only if for all $x\in U$ one has that $\{x\}\bar{\delta}(X\setminus U)$.

\begin{proposition}\label{close if and only if closures are close}
	Let $(X,\delta)$ be a proximity space, then for all $A,B\subseteq X$
	
	\[A\delta B\iff \overline{A}\delta\overline{B}\]
	
	where $\overline{A}$ and $\overline{B}$ denote closure within the topology induced by $\delta$.
\end{proposition}

\begin{definition}
	Given a topological space $X$, a proximity relation $\delta$ on $X$ is said to be {\bf compatible} with the topology on $X$ if the topology induced by $\delta$ is the original topology on $X$.
\end{definition}

\begin{proposition}\label{unique proximity on compact hausdorff spaces}
	If $X$ is a compact Hausdorff space, then there is a unique proximity on $X$ that is compatible with the topology on $X$. It is defined by:
	
	\[A\delta B\iff \overline{A}\cap\overline{B}\neq\emptyset\]
	
	where $\overline{A}$ denotes the closure of $A$ in $X$.
\end{proposition}

\begin{definition}
	If $A$ and $B$ are subsets of a proximity space $(X,\delta)$ then we say that $B$ is a $\delta$-{\bf neighbourhood} of $A$ if $A\bar{\delta}(X\setminus B)$. We denote this relationship by writing $A\ll B$.
\end{definition}

An elementary consequence of axiom $(5)$ of Definition \ref{proximity space definition} is that if $A$ and $B$ are subsets of a proximity space $(X,\delta)$ such that $A\ll B$ then there is a $C\subseteq X$ such that $A\ll C\ll B$. Moreover, it is an easy exercise to show that $A\ll B$ if and only if $\overline{A}\ll B$. That is, a set and its closure in the induced topology have the same $\delta$-neighbourhoods. A useful fact about $\delta$-neighbourhoods in proximity spaces whose topology is $T_{4}$ is the following:

\begin{lemma}\label{proximity neighbourhoods are actual neighbourhoods}
	Let $(X,\delta)$ be a separated proximity space whose induced topology is $T_{4}$. If $A,B\subseteq X$ are such that and $A\ll B$ then $A\subseteq int(B)$ where $int(B)$ denotes the interior of $B$ in $X$.
\end{lemma}
\begin{proof}
If $A\ll B$ then $\overline{A}\ll B$. By definition we then have that $\overline{A}\bar{\delta}(X\setminus B)$, which implies that $\overline{A}\bar{\delta}\overline{(X\setminus B)}$. Then $\overline{A}$ and $\overline{(X\setminus B)}$ are disjoint closed subsets of $X$. Because $X$ is $T_{4}$ there is an open set $U\subseteq X$ such that $\overline{A}\subseteq U$ and $U\cap\overline{(X\setminus B)}=\emptyset$. Therefore $A\subseteq U\subseteq B$, which is to say that $A\subseteq int(B)$.

\end{proof}
\vspace{\baselineskip}

With these initial basic definitions and results in hand we will proceed to stating the definitions and important results surround the proximity dimension $\delta d$. These definitions and results can be found in \cite{Smirnov}.

\begin{definition}
	Given a proximity space $(X,\delta)$ a finite cover $A_{1},\ldots, A_{n}$ of $X$ is called a $\delta$-cover if there is another finite cover $B_{1},\ldots, B_{n}$ of $X$ such that $B_{i}\ll A_{i}$ for $i=1,\ldots,n$. 
\end{definition}

\begin{definition}
	Given a proximity space $(X,\delta)$ the {\bf proximity dimension} of $X$, denoted $\delta d(X)$, is defined in the following way:
	\begin{enumerate}
		\item $\delta d(X)=-1$ if and only if $X=\emptyset$.
		\item For $n\geq 0$, $\delta d(X)\leq n$ if and only if every $\delta$-cover $\mathcal{U}$ can be refined by a $\delta$-cover of order at most $n+1$. 
		\item $\delta d(X)$ is the least integer $n$ such that $\delta d(X)\leq n$. If there is no such integer then $\delta d(X)=\infty$.
	\end{enumerate}
\end{definition}

\begin{theorem}\label{proximity dimension is covering dimension in compact hausdorff spaces}
	If $X$ is a compact hausdorff space, then $\delta d(X)=dim(X)$.
\end{theorem}

Note that there is no ambiguity in Theorem \ref{proximity dimension is covering dimension in compact hausdorff spaces} as Proposition \ref{unique proximity on compact hausdorff spaces} grants that there is only one possible interpretation of $\delta d$ on compact Hausdorff spaces.
\vspace{\baselineskip}

Next we proceed to Isbell's proximity inductive dimension. These definitions and results can be found in \cite{isbell}.
 
\begin{definition}\label{freeing set definition}
	Given a proximity space $(X,\delta)$ and two subsets $A,B\subseteq X$ such that $A\bar{\delta}B$, a subset $D\subseteq X$ is said to $\delta$-{\bf separate} $A$ and $B$, or is a $\delta$-{\bf separator} between $A$ and $B$, if $X\setminus D=U\cup V$ where $A\subseteq U,\,B\subseteq V,\,U\cap V=\emptyset,$ and $U\bar{\delta} V$. A subset $H\subseteq X$ is said to {\bf free} $A$ and $B$, or be a {\bf freeing set} for $A$ and $B$, if $H\bar{\delta}(A\cup B)$ and every $\delta$-neighbourhood of $H$ that is disjoint from $A\cup B$ is a $\delta$-separator between $A$ and $B$. 
\end{definition}

\begin{definition}\label{proximity inductive dimension definition}
	Given a proximity space $(X,\delta)$ the {\bf proximity inductive dimension} of $X$, denoted $\delta Ind(X)$, is defined inductively:
	\begin{enumerate}
		\item $\delta Ind(X)=-1$ if and only if $X=\emptyset$.
		\item For $n\geq 0$, $\delta Ind(X)\leq n$ if and only if for every pair of subsets $A,B\subseteq X$ such that $A\bar{\delta}B$ there is a set $H\subseteq X$ that frees $A$ and $B$ and is such that $\delta Ind(H)\leq n-1$.
		\item $\delta Ind(X)$ is the least integer $n$ such that $\delta Ind(X)\leq n$. If there is no such $n$, then $\delta Ind(X)=\infty$.
	\end{enumerate}
\end{definition}

\begin{proposition}\label{dense subspaces have larger inductive dimension}
	If $(Y,\delta)$ is a proximity space and $X\subseteq Y$ is a dense subspace, then $\delta Ind(X)\geq\delta Ind(Y)$.
\end{proposition}

\begin{proposition}\label{proximtiy inductive dimension is greater than proximity dimension}
	For every proximity space $(X,\delta)$, $\delta Ind(X)\geq\delta d(X)$.
\end{proposition}

We note that Proposition \ref{dense subspaces have larger inductive dimension} implies that in Definition \ref{proximity inductive dimension definition} we could have taken $A,B,$ and $H$ to be closed without altering the value of $\delta Ind$.
\vspace{\baselineskip}

\begin{definition}
	Let $X$ be a topological space and $A,B\subseteq X$ disjoint closed subsets. A closed set $C\subseteq X$ is called a {\bf separator} in $X$ between $A$ and $B$ if there are disjoint open sets $U,V\subseteq X$ such that $X\setminus C=U\cup V$,  with $A\subseteq U$ and $V\subseteq V$.
\end{definition}

The following result is an easy exercise, whose proof can be found in \cite{findimunifspaces2}.

\begin{proposition}
	Let $X$ be a compact Hausdorff space and $A,B\subseteq X$ disjoint closed subsets. If $C\subseteq X$ is a separator in $X$ between $A$ and $B$, then $C$ frees $A$ and $B$.
\end{proposition}

The converse to the above result is not true. 

\begin{example}\label{counterexample}
	Let $X$ be the "Topologist's Sine Curve". That is $X$ is the closure of the set
	
	\[A=\{(x,\sin(1/x))\in\mathbb{R}^{2}\mid x\in(0,1]\}\]

in $\mathbb{R}^{2}$. If we define $A=\{(0,-1)\}$, $B=\{(1,\sin(1))\}$, and $C=\{(0,1)\}$ then $C$ frees $A$ and $B$, but is not a separator between them.
\end{example}

\section{Brouwer Dimension}\label{brouwer dimension}

Here we will review the basic definitions and results surround the Brouwer dimension. For a brief history of the invariant see \cite{Brouwerdimoncompacthausdorff} or \cite{spacewithdifferingdimensions}.

\begin{definition}
	A {\bf continuum} is a nonempty compact connected Hausdorff space.
\end{definition}

In some places in the literature (such as \cite{nagami}) the word "compactum" is used for nonempty compact connected Hausdorff spaces. This is likely to distinguish the more general definition from the perhaps more standard definition of a continuum as a nonempty compact connected \emph{metric} space.

\begin{definition}\label{cut definition}
	Given a topological space $X$ and two disjoint closed subsets $A,B\subseteq X$, a closed subset $C\subseteq X$ that is disjoint from $(A\cup B)$ is called a {\bf cut} between $A$ and $B$ if every continuum $K\subseteq X$ such that $K\cap A\neq\emptyset$ and $K\cap B\neq\emptyset$ also satisfies $K\cap C\neq\emptyset$.
\end{definition}

It is an easy exercise to show that every separator in a topological space is also a cut. However as Example \ref{counterexample} shows, not every cut is a separator.

\begin{definition}
	Let $X$ be a $T_{4}$ topological space. The {\bf Brouwer dimension} of $X$, denoted $Dg(X)$, is defined inductively:
	\begin{enumerate}
		\item $Dg(X)=-1$ if and only if $X=\emptyset$.
		\item For $n\geq 0$, $Dg(X)\leq n$ if and only if for every pair of disjoint closed sets $A,B\subseteq X$ there is a cut $C\subseteq X$ between $A$ and $B$ such that $Dg(C)\leq n-1$.
		\item $Dg(X)$ is the least integer $n$ such that $Dg(X)\leq n$. If there is no such integer then $Dg(X)=\infty$.
	\end{enumerate}
\end{definition}

The following result appears in \cite{nagami} and will be used in the next section.

\begin{lemma}\label{important lemma}
	Let $X$ be a compact Hausdorff space, $A$ and $B$ disjoint closed subsets of $X$. If there is no connected set $K$ such that $K\cap A\neq\emptyset$ and $K\cap B\neq\emptyset$, then the empty set separates $A$ and $B$.
\end{lemma}

Said differently, this Lemma states that if the empty set is a cut between disjoint closed subsets of a compact Hausdorff space, then it is also a separator between them.

\section{Main Results}\label{main results}

In this final section we will prove that the proximity inductive dimension and the Brouwer dimension coincide on compact Hausdorff spaces. To do this we first characterize cuts within compact Hausdorff spaces.

\begin{proposition}\label{cut characterization}
	Let $X$ be a compact Hausdorff space and $A,B\subseteq X$ nonempty disjoint closed subsets. Then given a closed subset $C\subseteq X$ that is disjoint from $A\cup B$, the following are equivalent:
	\begin{enumerate}
		\item $C$ is a cut in $X$ between $A$ and $B$.
		\item Every closed neighbourhood of $C$ that is disjoint from $A\cup B$ is a separator between $A$ and $B$.
	\end{enumerate}
\end{proposition}
\begin{proof}
	Let $X,A,B,$ and $C$ be given. We may assume without loss of generality that $A$ and $B$ are nonempty. Otherwise every closed set disjoint from $A\cup B$ is a cut between $A$ and $B$.
	\vspace{\baselineskip}
	
	($(2)\implies(1)$) Assume that every closed neighbourhood $D$ of $C$ that is disjoint from $A\cup B$ is a separator between $A$ and $B$. If $C=\emptyset$ then $C$ is a closed neighbourhood of itself that is disjoint from $A$ and $B$, which would imply that the empty set is a separator between $A$ and $B$, which would imply that $C$ is a cut between $A$ and $B$. Assume then that $C\neq\emptyset$ and assume further towards a contradiction that $C$ is not a cut. Then there is a continuum $K\subseteq X$ such that $K\cap A\neq\emptyset$, $K\cap B\neq\emptyset$, but $K\cap C=\emptyset$. As $K$ and $C$ are disjoint closed sets there is a closed neighbourhood $D$ of $C$ that is disjoint from $K$. Then $K\subseteq X\setminus D=U\cup V$ where $U$ and $V$ are disjoint open sets of $X$ containing $A$ and $B$ respectively. This would imply that $K\cap U$ and $K\cap V$ are open subsets of $K$ that witness $K$ being disconnected, contradicting the connectedness of $K$. Therefore $C$ is a cut between $A$ and $B$.
	\vspace{\baselineskip}
	
	($(1)\implies(2)$) Now assume that $C$ is a cut between $A$ and $B$. Let $D$ be a closed neighbourhood of $C$ that is disjoint from $A$ and $B$. Then $C\subseteq int(D)$ and because $C$ is a cut between $A$ and $B$ we have that there is no connected set $K$ in the compact Hausdorff space $X\setminus int(D)$ that intersects both $A$ and $B$ nontrivially. Therefore by Lemma \ref{important lemma} the empty set is a separator in $X\setminus int(D)$ between $A$ and $B$. Then let $U$ and $V$ be disjoint open subset of $X\setminus int(D)$ that contain $A$ and $B$, respectively. Then $U^{\prime}:=U\cap (X\setminus D)$ and $V^{\prime}:=V\cap(X\setminus D)$ are disjoint open subsets of $X\setminus D$ and consequently disjoint open subsets of $X$ such that $X\setminus D=U^{\prime}\cup V^{\prime}$, $A\subseteq U^{\prime}$, and $B\subseteq V^{\prime}$. Therefore $D$ is a separator in $X$ between $A$ and $B$.
\end{proof}

\begin{proposition}\label{freeing sets are cuts}
	Let $X$ be a compact Hausdorff space, and $A,B\subseteq X$ disjoint closed sets. Given a closed subset $C\subseteq X$ that is disjoint from $A$ and $B$, the following are equivalent:
	\begin{enumerate}
		\item $C$ is a cut between $A$ and $B$.
		\item $C$ frees $A$ and $B$.
	\end{enumerate}
\end{proposition}
\begin{proof}
Let $X,A,B,$ and $C$ be given. As before we may assume without loss of generality that $A,B,$ and $C$ are nonempty as the result is trivial otherwise. 
\vspace{\baselineskip}

($(2)\implies(1)$) Assume that $C$ frees $A$ and $B$. Every closed neighbourhood $D$ of $C$ that is disjoint from $A$ and $B$ is a $\delta$-neighbourhood of $C$. Then by the definition of a freeing set we have that $X\setminus D=U\cup V$ where $A\subseteq U$, $B\subseteq V$, and $U\bar{\delta}V$. Then $U\bar{\delta} V$ in the subspace $X\setminus D$. Because $U\bar{\delta}V$ implies that $U$ and $V$ are disjoint we then have that $U$ and $V$ are open in $X\setminus D$, and are therefore open in $X$. We then in fact have that $D$ is a separator between $A$ and $B$. As $D$ was arbitrary we then have that $C$ is a cut between $A$ and $B$ by Proposition \ref{cut characterization}.
\vspace{\baselineskip}

($(1)\implies(2)$) Assume that $C$ is a cut between $A$ and $B$ and let $D\subseteq X$ be a $\delta$-neighbourhood of $C$ that is disjoint from $A\cup B$. We may assume that $D$ is an open neighbourhood of $C$ because if $C\ll D$ then $C\subseteq int(D)$ by Lemma \ref{proximity neighbourhoods are actual neighbourhoods} and if a subset of $D$ is a $\delta$-separator then $D$ is as well. We then let $D^{\prime}$ be a closed neighbourhood of $C$ such that $C\subseteq D^{\prime}\subseteq D$ and $D^{\prime}\cap(X\setminus D)=\emptyset$. Then $D^{\prime}$ is a closed neighbourhood of $C$ that is disjoint from $A$ and $B$, so by Proposition \ref{cut characterization} $D^{\prime}$ is a separator in $X$ between $A$ and $B$. Then let $U$ and $V$ be disjoint open subsets of $X$ so that $X\setminus D^{\prime}=U\cup V$, $A\subseteq U$, and $B\subseteq V$. We then consider

\[U^{\prime}=U\cap(X\setminus D)\text{ and } V^{\prime}:=V\cap(X\setminus D)\]

and claim that $U^{\prime}\bar{\delta}V^{\prime}$. To see this we note that because $U^{\prime}$ and $V^{\prime}$ are subsets of the closed set $X\setminus D$ we must have that $\overline{U^{\prime}}$ and $\overline{V^{\prime}}$ are also subsets of $X\setminus D$. Then $\overline{U^{\prime}}\cap\overline{V^{\prime}}\subseteq X\setminus D$. However, as $U^{\prime}\subseteq U$ and $V^{\prime}\subseteq V$ we must have that $\overline{U^{\prime}}\subseteq\overline{U}$ and $\overline{V^{\prime}}\subseteq\overline{V}$. Therefore we have

\[\overline{U^{\prime}}\cap\overline{V^{\prime}}\subseteq X\setminus D\text{ and }\overline{U^{\prime}}\cap\overline{V^{\prime}}\subseteq\overline{U}\cap\overline{V}\]

However, $\overline{U}\cap\overline{V}$ is a subset of $D^{\prime}$ and $D^{\prime}\cap(X\setminus D)=\emptyset$. Therefore we have that $\overline{U}^{\prime}\cap\overline{V}^{\prime}=\emptyset$, which gives us that $U^{\prime}\bar{\delta}V^{\prime}$. In summary, $X\setminus D$ is the union of the disjoint sets $U^{\prime}$ and $V^{\prime}$ that contain $A$ and $B$ respectively, and are not close. That is, $D$ is a $\delta$-separator between $A$ and $B$, which implies that $C$ frees $A$ and $B$.
\end{proof}

\begin{theorem}\label{proximity inductive is equal to brouwer dimension on compact hausdorff spaces}
	For every compact Hausdorff space, $Dg(X)=\delta Ind(X)$.
\end{theorem}
\begin{proof}
We will show that $\delta Ind(X)\geq Dg(X)$ by induction on $\delta Ind(X)$. The result is obvious when $\delta Ind(X)=-1$. Assume then that the result holds for $\delta Ind(X)<n$ for $n\geq 0$ and assume that $\delta Ind(X)=n$. If $A$ and $B$ are disjoint closed subsets of $X$ then there must be a closed set $C\subseteq X$ that frees $A$ and $B$ and satisfies $\delta Ind(C)\leq n-1$. By Proposition \ref{freeing sets are cuts} $C$ is a cut between $A$ and $B$, and the inductive hypothesis gives us that $Dg(C)\leq\delta Ind(C)\leq n-1$. Therefore $Dg(X)\leq n$. Clearly, if $\delta Ind(X)=\infty$ then $Dg(X)\leq\delta Ind(X)$. Therefore $\delta Ind(X)\geq Dg(X)$. The argument showing that $Dg(X)\geq\delta Ind(X)$ is a similar induction argument. 
\end{proof}

In \cite{spacewithdifferingdimensions} a compact Hausdorff space $B$ was constructed with the property that $Dg(X)=3$ and $dim(X)=2$. Then the conjunction of Theorem \ref{proximity inductive is equal to brouwer dimension on compact hausdorff spaces} and Theorem \ref{proximity dimension is covering dimension in compact hausdorff spaces} grants us the following corollary.

\begin{corollary}
	There is a compact Hausdorff space $B$ such that $\delta Ind(B)=3$ and $\delta d(B)=2$.
\end{corollary}

\bibliographystyle{abbrv}
\bibliography{Boundaries_of_coarse_proximity_spaces_and_boundaries_of_compactifications}{}

\end{document}